\documentclass[12pt,reqno]{amsart}
\usepackage{geometry}
\usepackage[latin1]{inputenc}
\usepackage[italian,english]{babel}
\usepackage{amsmath, amsfonts, amsthm}
\usepackage{paralist}
\numberwithin{equation}{section}
\newtheorem{thm}{Theorem}[section]

\newtheorem{prop}[thm]{Proposition}

\newtheorem{defn}[thm]{Definition}
\theoremstyle{definition}
\newtheorem{rem}[thm]{Remark}
\theoremstyle{remark}

\newcommand{\ds}{\displaystyle}

\newcommand{\R}{\mathbb{R}}

\newcommand{\de}{\partial}

{\left\{\begin{array}{@{}l@{}}}{\end{array}\right.}
\patchcmd{\abstract}{\scshape\abstractname}{\textbf{\abstractname}}{}{}
\makeatletter %note a di pagina senza numero 1
\def\@makefnmark{} %note a di pagina senza numero 2
\makeatother %note a di pagina senza numero 3

\title{An estimate for the anisotropic maximum  curvature in the planar case} 
\author[G. Paoli]{
	Gloria Paoli}
\address{Dipartimento di Matematica e Applicazioni ``R. Caccioppoli'', Universit\`a degli studi di Napoli Federico II \\ Via Cintia, Complesso Universitario Monte S. Angelo, 80126 Napoli, Italy.}
\email{gloria.paoli@unina.it}
\begin{document}
\maketitle
%\markright{\tiny Sharp estimates for the first $p$-Laplacian eigenvalue and for the torsional rigidity on convex sets with  holes}
%\markright{SHARP ESTIMATES FOR THE FIRST $p$-LAPLACIAN EIGENVALUE}
\begin{abstract}
We fix a Finsler norm $F$ and, using the anisotropic curvature  flow, we prove that the anisotropic maximum curvature $k^F_{\max}$  of a smooth Jordan curve is such that $ k^F_{\max}(\gamma)\geq \sqrt{\kappa/A}$ , where $A$ is the  area enclosed by $\gamma$ and $\kappa$ the area of the unitary Wulff shape associated to the anisotropy $F$.

\end{abstract}

\section{Introduction}

Let $\Omega\subset\mathbb{R}^2$, be a bounded, connected, $C^2$ smooth domain and let us set $\gamma:=\de\Omega$, so that $\gamma$ is  a smooth Jordan curve. In \cite{PI} the following inequality is proved:
\begin{equation}\label{pank}
k_{\max}(\gamma)\geq \sqrt{\dfrac{\pi}{A(\Omega)}},
\end{equation}
where   $k_{max}(\gamma)$ is  the maximum curvature of $\gamma$  and $A(\Omega)$ is the area enclosed by $\gamma$; moreover, equality holds if and only if $\gamma$ is a circle. 
  Since the original work \cite{PI} is hardly available, we refer the reader to \cite{HT} for the proof of \eqref{pank}.

In \cite{P} the author provides a new proof of the inequality \eqref{pank} by means of the curve shortening flow. Recalling the definition, we have that 
a  family $u:\mathbb{S}^1\times [0,T]\to\mathbb{R}^2$ of smooth Jordan curves flows by anisotropic curvature flow if 
\begin{equation}\label{anis_curv_flow}
\dfrac{\partial u(\theta,t)}{\partial t}=  -k(\theta,t)\nu(\theta,t),
\end{equation}
where $\nu(\theta,t)$ and $k(\theta,t)$ are respectively the outer unit normal and the curvature of the curve $u(\cdot,t)$ at the point $u(\theta,t)$. % We recall that the limiting shape is a round point, with convergence in $C^\infty$ norm and there exists $\tau\in[0,T)$ such that $u(\cdot, t)$ is convex for $t\in[\tau,T)$.
 For some reference about the curve shortening flow and its properties see, for example, \cite{GH, G}.

The purpose of the present paper is to find the analogous result of \eqref{pank}  in the anisotropic case. More precisely, let $F:\mathbb{R}^2\to[0,+\infty)$ be a Finsler norm; we denote by 
$$\mathcal W = \{  \xi \in  \R^2 \colon F^o(\xi)< 1 \},$$
 the  unit Wulff shape centered at the origin and we set
$\kappa :=A(\mathcal{W})$. For  every $x\in\partial \Omega$, 
	the $F$-anisotropic  curvature is defined as $$k^F_{\de \Omega}(x)= {\rm div}\left(n^F_{\de \Omega}(x)\right),$$
where $n^F_{\de \Omega}(x)=\nabla F(\nu_{\de \Omega}(x))$ is the anisotropic normal, while $\nu_{\de \Omega}(x)$ is the Euclidean outer unit normal. Then, we  denote by $k^F_{max}(\de \Omega)$ the  maximum curvature over $\de \Omega$, that is
	$$ k^F_{max}(\de \Omega):=|| k^F_{\de \Omega}||_{L^{\infty}(\de \Omega)}.$$
%	\max_{x\in\de \Omega}k^F_{\de \Omega}(x).$$
The main result of this work is the following.

\vspace{1pt}
{\bf Main Theorem.} {\it 	Let  $\Omega\subseteq \mathbb{R}^2$ such that  $\gamma:=\de \Omega$  is a  smooth Jordan curve. Then, 
	\begin{equation}\label{mainn}
			k^F_{\max}(\de\Omega)\geq 	k^F_{\max}(\de\mathcal{W}^*),
	\end{equation}
	where $\mathcal{W}^*$ is a  Wulff shape having the same area as $\Omega.$ Moreover, equality holds if and only if $\Omega$ coincides with a Wulff shape.
	.}
\vspace{1pt}\\
Equivalently, the result in \eqref{mainn} can be restated in the following  form:
	\begin{equation}\label{main}
	k^F_{\max}(\gamma)\geq \sqrt{\dfrac{\kappa}{A(\Omega)}}.
	\end{equation}
Section $3$ is dedicated to the proof of the Main Theorem. The scheme of the proof is close to the one used in \cite{P} and it is based on the use of the  anisotripic flow.
We recall that a  family $u:\mathbb{S}^1\times [0,T]\to\mathbb{R}^2$ of smooth Jordan curves flows by anisotropic curvature flow if 
\begin{equation}\label{anis_curv_flow}
\dfrac{\partial u(\theta,t)}{\partial t}=  \left(F(\nu(\theta,t))\:k^F(\theta,t)\right)\nu(\theta,t),
\end{equation}
where $k^F(\theta,t)$ is the anisotropic curvature of the curve $u(\cdot,t)$ at the point $u(\theta,t)$. For some reference see, for example,  \cite{A,BP,CZ,MNP}.
In the proof we will reduce our study to the case of convex  curves   and we will   use the so called Wulff- Gage inequality. This inequality, proved in \cite{GO},  states  that, if $K\subseteq\mathbb{R}^2$ is a convex set, then
\begin{equation}\label{Wulff Gage inequality}
\ds\int_{\de K} (k^F_{\de K}(x)^2F(\nu_{\de K}(x) ) \;d\mathcal{H}^1(x)\geq  \dfrac{\kappa P_F( K)}{A(K)},
\end{equation}
where $P_F(K)= \int_{\de \Omega}F(\nu_{\de \Omega}(x)) \, d \mathcal H^{1}(x)$ is the anisotropic perimeter of $K$. The isotropic version of this inequality was proved in \cite{GH} for convex sets of the plane and generalized in \cite{BH, FKN1, FKN2} for non convex sets, whose boundary is simply connected.

We point  out that in \cite{PP}  the authors show that  inequality \eqref{pank} can be generalized in higher dimensions if we restrict to the class of sets which are starshaped;  in this case balls still achieve the minimal maximal mean curvature among domains with the same volume. However, if we remove the additional topological constraint of starshapedness and consider bounded smooth domains with a connected boundary the result, as showed in \cite{FNT},  is no longer true for $n>3$. 

Moreover in \cite{PP}  the problem of minimizing   the maximal curvature is linked to an estimate  of the Laplacian eigenvalue problem with Robin boundary conditions as the boundary parameter $\alpha$ goes to $-\infty$.
Let $\Omega$ be a bounded,  open subset of $\mathbb{R}^n$, $n\geq2$, with Lipschitz boundary; its Robin eigenvalues related to the Laplacian are the real numbers $\lambda$ such that   
\begin{equation}\label{rob}
\begin{cases}
-\Delta u=\lambda u &\mbox{in}\ \Omega\\[.2cm]
\frac{\de u}{\de \nu}+\alpha u=0&\mbox{on}\ \de\Omega
\end{cases}
\end{equation}
admits non trivial $W^{1,2}(\Omega)$ solutions; $\alpha$ is an arbitrary  real constant, which will be referred to as boundary parameter of the Robin problem. %For each fixed $\Omega $ and $\alpha$ there is a sequence of eigenvalues 
%$$\lambda_1(\alpha,\Omega)\leq\lambda_2(\alpha,\Omega)\leq\dots\rightarrow+\infty $$
%which depend on $\alpha$.
 In particular, the first non trivial Robin eigenvalue of $\Omega$ is characterized by the expression
\begin{equation*}\label{var_char}
\lambda_1(\alpha , \Omega)=\min_{\substack{u\in W^{1,2}(\Omega) \\ u\neq 0}}\dfrac{\ds\int_{\Omega}\left\vert Du\right\vert^2\;dx+\alpha\ds\int_{\de\Omega}|u|^2\;d\mathcal{H}^1}{\ds\int_{\Omega}|u|^2\;dx}.
\end{equation*}
Let us now assume that $\alpha<0$ and $\Omega\subset\mathbb{R}^n$ is a bounded and Lipschitz domain. 
If we put a constant function as a test function in the Rayleigh quotient above, we find out that the first eigenvalue is always strictly negative. In 1977 Bareket conjectured that the maximizer of the first Robin eigenvalue with negative parameter among sets with the same volume  was a ball \cite{Ba}. However in \cite{FK} the authors disproved Bareket's conjecture, showing that the first Robin-Laplacian computed on a spherical shell is asymptotically  greater than the one computed on a ball with the same volume.  In \cite{PP} this was clarified by  showing that for $\Omega\subseteq
\mathbb{R}^n$ of class $C^{1,1}$, then the following two-terms asymptotics holds
\begin{equation}\label{pank_asy}
 \lambda_1(\Omega,\alpha)=-\alpha^2-(n-1)\alpha \sup_{\de\Omega} H+o(\alpha^{2/3}),   \end{equation}
as $\alpha\to-\infty$, where $H$ is the mean curvature of the boundary, that is a  generalisation of the curvature in higher dimension. 
 We recall that  in  \cite{FK},  it is  proved that Bareket's conjecture holds for $\alpha$ negative small enough in absolute value.  
%More precisely, the authors  proved that, for bounded planar domains of class $C^2$ and fixed  area,  there exists a negative number $\alpha_*$, depending only on the area, such that \eqref{dis} holds for all $\alpha\in[\alpha_*,0]$.
%when the boundary parameter goes to $+\infty$. 

Our inequality can possibly  have an  application in the study of the anisotropic counterpart of the Robin problem, that  is the problem  
 \begin{equation*}
\begin{cases}
-{\rm div}\left(F(Du)F_{\xi}(Du)\right)=\lambda_F(\alpha , \Omega)u &\mbox{in}\ \Omega\\[.2cm]
\langle F(Du)F_{\xi}(Du),\nu_{\de\Omega}\rangle+\alpha F(\nu_{\de\Omega})u=0 &\mbox{on}\ \de\Omega,
\end{cases}
\end{equation*}
where $F$ is a fixed Finsler norm,that has been studied for example in \cite{DG,GT,PT}, and  with the following variational characterization of the first eigenvalue:
 \begin{equation*}\label{var_char}
 \lambda_{1, F}( \alpha , \Omega )=\min_{\substack{u\in W^{1,2}(\Omega) \\ u\neq 0}}\dfrac{\ds\int_{\Omega}F^2(Du)\;dx+\alpha\ds\int_{\de\Omega}|u|^2F(\nu_{\de\Omega})\;d\mathcal{H}^1(x)}{\ds\int_{\Omega}|u|^2\;dx}.
 \end{equation*}
A possible future direction of investigation could be the generalization  of  inequality \eqref{pank_asy} for the study of the anisotropic  Robin problem.

\section{Preliminaries}
\subsection{Notation}
In the following we  will denote by $\langle\cdot,\cdot\rangle$ the standard Euclidean scalar product  and by $|\cdot|$ the Euclidean norm in $\mathbb{R}^2$.
We denote by $\mathcal{H}^1$
the $1-$dimensional Hausdorff measure in $\mathbb{R}^2$. %and by $V(\)$ the Lebesgue measure in $\mathbb{R}^2$.
If $\Omega $ is a set of $\mathbb{R}^2$ with  Lipschitz boundary, for $\mathcal{H}^{1}$-almost every $x\in\partial \Omega$, $\nu_{\de \Omega}(x)$ is the outward unit Euclidean normal to $\partial \Omega$ at $x$. Moreover,  $A(\Omega)$ is  the area of the set $\Omega$, i.e. its Lebesgue  measure in $\mathbb{R}^2$.%and $P(\Omega)$ its perimeter, i.e. its $\mathcal{H}^{1}$-measure.

\subsection{Finsler norm: definitions and some properties}
Let $F:\mathbb{R}^2\to[0,+\infty)$ be a convex function such that for some constant $a>0$
\begin{equation}
\label{eq:lin}
a|\xi| \le F(\xi),\quad \xi \in  \R^{2}
\end{equation}
and
\begin{equation}
\label{eq:omo}
F(t\xi)=|t|F(\xi), \quad t\in \R,\,\xi \in  \R^{2}.
\end{equation}
 These hypotheses on $F$ imply that there exists $b\ge a$ such that
\begin{equation*}
\label{upb}
F(\xi)\le b |\xi|,\quad \xi \in  \R^{2}.
\end{equation*}
Moreover, throughout this paper, we will assume that $F^2$ is strongly convex, that is $F\in C^{2}(\mathbb R^{2}\setminus \{0\})$ and that the Hessian matrix $\nabla^2_\xi F^2$ is positive definite in $\mathbb{R}^2\setminus	\{0\}$. %This implies that 
 %$F$ is elliptic, i.e. there exists $A>0$ such that 
% $$ \nabla^2_\xi\left( F^2\right) \geq A \;Id$$
%in the distributional sense. 
Under these assumptions,  $F$ is called an elliptic norm.
The polar function $F^o\colon \R^2 \rightarrow [0,+\infty[$ 
of $F$ is defined as
\begin{equation*}
F^o(v)=\sup_{\xi \ne 0} \frac{\langle \xi, v\rangle}{F(\xi)}
\end{equation*}
and it  is easy to verify that also $F^o$ is a convex function
that satisfies properties \eqref{eq:omo} and
\eqref{eq:lin}. $F$ and $F^o$ are usually called Finsler norm. Furthermore, it holds 
\begin{equation*}
F(v)=\sup_{\xi \ne 0} \frac{\langle \xi, v\rangle}{F^o(\xi)},
\end{equation*}
 which implies the following anisotropic version of the Cauchy Schwartz inequality
\begin{equation*}
\label{imp}
|\langle \xi, \eta\rangle| \le F(\xi) F^{o}(\eta), \qquad \forall \xi, \eta \in  \R^{2}.
\end{equation*}
We introduce now the following notations:
$$\mathcal W = \{  \xi \in  \R^2 \colon F^o(\xi)< 1 \}$$
is called the  unit Wulff shape centered at the origin and we put
$\kappa=V(\mathcal W)$.  Moreover, we denote by $\mathcal W_r(x_0)$
the set $r\mathcal W+x_0$, that is the Wulff shape centered at $x_0$
with measure $\kappa r^2$, so that  $\mathcal W_r(0)=\mathcal W_r$.
We observe that the strong convexity of $F^2$ implies that $\mathcal{W}$ is strictly convex and this ensures that $F^o\in C^1(\mathbb{R}^2\setminus\{0\})$. More precisely, we have that the strict convexity of the level sets of $F$ is equivalent to the continuous differentiability of $F^o$ in $\mathbb{R}^2\setminus\{0\})$; for more details see \cite{S}.
We conclude this paragraph recalling some useful  properties of $F$ and $F^o$:
\begin{gather*}
\label{prima}
\langle \nabla F(\xi) , \xi \rangle= F(\xi), \quad  \langle\nabla F^{o} (\xi), \xi \rangle
= F^{o}(\xi),\qquad \forall \xi \in
\R^2\setminus \{0\};
\\
\label{seconda} F(  \nabla F^o(\xi))=F^o( \nabla F(\xi))=1,\quad \forall \xi \in
\R^2\setminus \{0\};
\\
\label{terza} 
F^o(\xi)   \nabla F( \nabla F^o(\xi) ) = F(\xi) 
\nabla F^o\left(  \nabla F(\xi) \right) = \xi\qquad \forall \xi \in
\R^2\setminus \{0\}. 
\end{gather*}

\subsection{Anisotropic perimeter }
\noindent
In the following we are fixing a Finsler norm $F$. 

\begin{defn}
	Let $\Omega$ be a bounded open subset of $\R^2$ with Lipschitz boundary, the anisotropic perimeter of $\Omega$ is defined as 
	\[
	P_F(\Omega)=\displaystyle \int_{\de \Omega}F(\nu_{\de \Omega}(x)) \, d \mathcal H^{1}(x),
	\]
	where $\nu_{\de \Omega}$ is the  Euclidean outer normal to $\partial\Omega$ defined almost everywhere.
\end{defn}
Clearly, the anisotropic perimeter of  $\Omega$ is finite if and only if the  Euclidean perimeter of $\Omega$, that we denote by $P(\Omega)$, is finite. Indeed, by the quoted properties of $F$ we have  that
$$
aP(\Omega) \le P_F(\Omega) \le bP(\Omega).
$$

\noindent Moreover the following isoperimetric inequality is proved for the anisotropic perimeter, see for istance \cite{ALFT,Bu,DG,DP, FM}.
\begin{thm}
	Let $\Omega$ be a subset of $\mathbb{R}^2$ with finite perimeter. Then, 
	$$
	P_F(\Omega) \ge 2 \kappa^\frac{1}{2} {A\left(\Omega\right)}^{\frac{1}{2}},
	$$
where $A(\Omega)$ is the area of $\Omega$.
	Equality holds if and only if $\Omega$ is homothetic to a Wulff shape.
\end{thm}
\subsection{Anisotropic curvature}
Since the main result of this paper  concerns set of $\mathbb{R}^2$ with $C^2$ boundary, from now on we will restrict our study  to this class of sets.% the following definitions can be easily generalized for sets with Lipschitz boundary.
\begin{defn}
Let   $\Omega$ be an open, bounded  subset of $\mathbb{R}^2$ with $C^2$  boundary. At each point of $\de \Omega$, we define 
	the $F$-normal vector:  $$n^F_{\de \Omega}(x)=\nabla F(\nu_{\de \Omega}(x)),$$
	sometimes called the \textit{Cahn-Hoffman} field.
\end{defn}
In particular, we observe that, by the properties of $F$, we have that 
\begin{equation}\label{is_one}
	F^o(n^F_{\de \Omega})=1.
\end{equation}
We now give the definitions of anisotropic curvature and of anisotropic maximum curvature.
\begin{defn}
	Let   $\Omega\subset\mathbb{R}^2$ be an open, bounded set  with $C^2$ boundary. For  every $x\in\partial \Omega$, we define  
	the $F$-anisotropic  curvature as $$k^F_{\de \Omega}(x)= {\rm div}\left(n^F_{\de \Omega}(x)\right).$$
	Moreover, we denote by $k^F_{max}(\de \Omega)$ its  maximum over $\de \Omega$, that is
	$$ k^F_{max}(\de \Omega):=  || k^F_{\de \Omega}||_{L^{\infty}(\de \Omega)}. $$%\max_{x\in\de \Omega}k^F_{\de \Omega}(x).$$%=||k^F_{\de \Omega}||_{L^\infty(\de\Omega)}.$$
\end{defn}
We recall that for a Wulff shape of the form $\frac{1}{\lambda}\mathcal{W}\subset\mathbb{R}^2$, with $\lambda>0$,  we have that  (for the details of the computation see \cite{BP}), for every $x\in\de\left(\frac{1}{\lambda}\mathcal{W}\right)$, 
\begin{equation*}
k^F_{\de K}(x)=\lambda.
\end{equation*}
Finally, in order to prove  our main theorem, we will need the following result related to  the anisotropic curvature of a convex set, whose proof  can be found in \cite{GO} (see Theorem $0.7$). \
 %the following result is proved %(observing that the polar function $F$ is the support function of the Wulff shape $\mathcal{W}$).
\begin{prop}[Wulff-Gage inequality]
Let $K\subseteq\mathbb{R}^2$ be a bounded convex set with $C^2$ boundary.
Then, 	
\begin{equation}\label{Wulff Gage inequality}
\ds\int_{\de K} (k^F_{\de K}(x)^2F(\nu_{\de K}(x) ) \;d\mathcal{H}^1(x)\geq  \dfrac{\kappa P_F( K)}{A(K)}
\end{equation}
and there is equality if and only if $K$ is a Wulff shape.

\end{prop}

\subsection{Anisotropic curvature flow}
Throughout this paper, we will use the following notations. We consider a family of closed  curves $u=u(s,t):\mathbb{S}^1\times [0,T]\to\mathbb{R}^2$, where  $s$  the arc-length parameter and we use the conventional notation $\partial_s(u(s,t))=u_s(s,t)$.
Moreover,  $\tau(s,t)=u_s(s,t)=\left(\sin\left(\theta(s,t)\right) ,-\cos \left(\theta(s,t)\right)\right)$ will be the unit tangent  and  $\nu(s,t)=\left(\cos\left(\theta(s,t)\right), \sin\left(\theta(s,t)\right)\right)$ the unit normal of $u$; $\theta =\theta(s,t)$ is called the normal angle (determined modulo $2\pi$) and we may use it to parametrize the curve $u(\cdot,t)$.
%Whenever no confusion is possible, we shall write $\tau$, $\nu$ and $k$ in place of $ \tau_u$, $ \nu_u$ and $k_u$, where with $k_u$ we denote the scalar curvature.%Throughout this paper we will use the following notations. We consider a family of closed  curves $u=u(s,t)=\left(x(s,t),y(s,t)\right):\mathcal{S}^1\times [0,T]\to\mathbb{R}^2$, denoting  by  $s$  the arc-length parameter, having that  $\partial_s(\cdot)=\partial_x(\cdot)/|u_x|$. Moreover,  $\tau_u=u_x/|u_x|=u_s=(\sin \theta,-\cos \theta)$ will be the unit tangent and  $\nu_u=(\cos\theta, \sin\theta)$ the unit normal of $u$; $\theta =\theta(s)$ is called the normal angle, it is determined modulo $2\pi$ and we may use it to parametrize the curve.
 %Whenever no confusion is possible, we shall write $\tau$, $\nu$ and $k$ in place of $ \tau_u$, $ \nu_u$ and $k_u$, where with $k_u$ we denote the scalar curvature.
The classical Frenet formulas assert that 
\begin{equation}\label{firs_frenet}
u_{ss}(s,t)=\tau_s(s,t)=k (s,t)\nu(s,t), 
\end{equation}
\begin{equation}\label{second_frenet}
\nu_s(s,t)=-k (s,t)\tau(s,t),
\end{equation}
where $k$ is the scalar curvature.
Another usefull relation is the following
\begin{equation}\label{normal_angle}
	k(s,t)=\theta_s(s,t).
\end{equation}
Finally, we recall the definition of support function  (see for istance \cite{S}). Let $\gamma:\mathcal{S}^1\to \mathbb{R}^2$ be  a smooth Jordan  curve,
 %Let $u:\mathcal{S}^1\times[0,T]\to\mathbb{R}^2$ be  a smooth Jordan  curve. Fixing $t\in[0,T]$, we are considering  a curve $u(\cdot,t)$ of the above family
  let us take the normal angle $\theta$ as  parameter for $\gamma$ and let us  denote its components by  $\gamma(\theta)=\left(x(\theta),y(\theta)\right)$. The support function associated to $\gamma$ is defined as 
  \begin{equation*}
  h(\theta):=\langle \left(x(\theta),y(\theta)\right),\left(\cos\theta,\sin\theta\right)\rangle .
  \end{equation*}
If we  denote by $'$ the derivative  with respect to $\theta$, we have that 
\begin{equation*}
h'(\theta)=-x(\theta)\sin(\theta)+y(\theta)\cos(\theta).
\end{equation*}
Therefore, $\gamma$ can be recovered from $h$ by 
$$x(\theta)=h(\theta) \cos(\theta)-h'(\theta)\sin(\theta),$$
$$y(\theta)=h (\theta)\sin (\theta)+h'(\theta)\cos(\theta).$$
%We recall here the definition of support function (see for istance \cite{S}).
%\begin{defn}
%	Let $E\subset\mathbb{R}^2$ be a closed convex set with $\emptyset\neq E\neq\mathbb{R}^2$. The support function of $E$ is defined by
	%$$h(\theta):=\sup \{ \langle x,\theta\rangle\;:\;x\in E   \} \quad {\rm for}\;\; \theta\in\mathbb{S}^1,$$
%	where $\langle\cdot,\cdot\rangle$ denotes the scalar product in $\mathbb{R}^2$.
%\end{defn}
% $h(\theta):=\langle \left(x(\theta),y(\theta)\right),\left(\cos\theta,\sin\theta\right)\rangle$, 
%we have that 
%$$x_u(\theta)=h \cos(\theta)-h_\theta'\sin(\theta),$$
%$$y_u(\theta)=h \sin (\theta)+h_\theta'\cos(\theta),$$
%where $h_\theta$ is the derivative of $h$ with respect to $\theta$.%support function, come si scrivono le componenti anche weinstock
We now give the definition of the anisotropic flow; for more details and for the proofs of the properties below see for istance \cite{CZ}. 
In the following, whenever no confution is possible, we shall write $ \tau$, $ \nu$ and $k$ as referred to $u$, using a notation that will not account for the choice of the curve, otherwise we will specify the curve to which they are referred.
\begin{defn}
The family $u:\mathbb{S}^1\times [0,T]\to\mathbb{R}^2$ of smooth Jordan curves flows by anisotropic curvature flow if 
\begin{equation}\label{anis_curv_flow}
\dfrac{\partial u(s,t)}{\partial t}=  \left(F(\nu(s,t))\:k^F(s,t)\right)\nu(s,t).
\end{equation}
\end{defn}
In the following two remarks we recall some important properties of the anisotropic curvature flow that we will use later.
\begin{rem}\label{regulari}
We observe 	that, since the curve $u$ is smooth and the anisotropy $F$ is elliptic,  we can write the anisotropic curvature as
\begin{equation}\label{regcur}
k^F(s,t)=\left(\nabla^2F(\nu(s,t))\tau(s,t)\cdot\tau(s,t)\right)k(s,t).
\end{equation}
Consequently, we have that the anisotropic curvature is controlled from above and from below by the Euclidean curvature. 
\end{rem}

\begin{rem}\label{convexification }
	
If we consider  a family of curves $u(\cdot,t )$ flowing by  anisotropic curvature flow,  we have that the limiting shape is a round point and that  there exists a time $\bar{t}\in [0,T)$	such that $u(\cdot,t)$ is convex for $t\in [\bar{t},T)$, even though the initial curve is not convex.
For a proof of this fact see, for istance, \cite{CZ, CZ2,GL}.
\end{rem}
%$K^F_{u}(\theta):=K^F_u(x_u(\theta),y_u(\theta)$.
As observed in \cite{MNP}, we can rewrite the anisotropic flow as follows. For semplicity of notation,  in the following formulas,  we will not mention the dependence from $s$ and $t$. So, let us define 
\begin{equation*}
\phi(\theta):= F(\nu)=F(\cos \theta,\sin\theta)
\end{equation*}
and let us  observe that,  by the divergence theorem,  $k^F=\left(\nabla_\xi^2\left(F^o(\nu)\right)\tau\cdot\tau\right) k$. Since  we have  $F^o(\theta)+\left(F^o(\theta)\right)''=\nabla_\xi^2\left(F^o(\nu)\right)\tau\cdot\tau$, then 
\begin{equation}\label{flux_2}
	u_t=\psi(\theta) k \nu,
\end{equation}
where
\begin{equation}\label{phi_def}
\psi(\theta):=\phi(\theta)\left(\phi(\theta)+\phi''(\theta)\right).
\end{equation}
In particular, the proof of the following result  can be found in \cite{MNP} (proof of Proposition $1$).

\begin{prop}
It holds 
\begin{equation}\label{equazione del calore anisotropa}
\left(\partial_t-\psi\partial_{ss}\right) \frac{\left(k^F\right)^2}{2}\leq \left(  3 k h \phi'+h' k \phi   \right)\partial_s(k^F)^2+(k^F)^4,
\end{equation}
wher $h=\phi+\phi''$.
\end{prop}
In \cite{CZ} can be found the computation of the first derivative of the area enclosed by a family of curves that flows by the   anisotropic curvature flow (see the following Proposition). More precisely, the first derivative is  proved integrating by parts the  formula that gives the area enclosed by a curve $\gamma$, that is 
\begin{equation*}
A(\gamma)=-\frac{1}{2}\ds\int_{\gamma}\langle\gamma,\nu\rangle\; ds.
\end{equation*}

\begin{prop}\label{derivative area}
Let  $u:\mathbb{S}^1\times [0,T]\to\mathbb{R}^2$ a family of smooth Jordan curvan satisfying \eqref{anis_curv_flow}. If we denote by $u_t(\cdot):=u(\cdot,t)$ and by $A(t)$ the area enclosed by $u_t$,  we have 
\begin{equation}\label{derar}
\frac{dA(t)}{dt}=-\ds\int_{u_t}  F(\nu_{u_t}(s,t)) k^F_{u_t}(s,t) ds,
\end{equation}
where $\nu_{u_t}$ and $k^F_{u_t}$ are respectively the unit normal  and the anisotropic curvature of the curve $u_t$.
\end{prop}

% Fare un remark su questa cosa?????

\section{Main result and its proof}

\begin{thm}
Let  $\Omega\subseteq \mathbb{R}^2$ such that  $\gamma:=\de \Omega$  is a  smooth Jordan curve. Then,
\begin{equation}\label{main}
k^F_{\max}(\gamma)\geq \sqrt{\dfrac{\kappa}{A(\Omega)}}
\end{equation}
and there is  equality if and only if $\Omega$ coincides with a Wulff shape.
\end{thm}

\begin{proof}

\underline{\textbf{Step 1: Uniqueness} }
Using a stardard argument we  prove that, if  inequality \eqref{main} is proved, then  equality holds only for Wulff shapes. Let assume that \eqref{main} is  true and, by contradiction,  that the equality holds for a curve $\gamma$ that is not the boundary of a Wulff shape. 
\noindent
Thus, there exists a point  $x\in \gamma$ such that  $k^F_{\de K}(x)\leq k^F_{\max}(\gamma)$. By a small local deformation around $x$, we can construct a smooth Jordan curve $\gamma'$ such that the following two conditions hold
\begin{itemize}
	\item  $ k^F_{\max}(\gamma')= k^F_{\max}(\gamma)$,
	\item the area  $A'$  enclosed by $\gamma'$ is strictly smaller than the area $A$ enclosed by $\gamma$.
\end{itemize}
In this way we have a contradiction, since $$k^F_{\max}(\gamma')<\sqrt{\kappa/A'}.$$
\noindent
\underline{\textbf{Step 2: The inequality holds for convex  curves} }
\noindent
Let us  assume that $\gamma$ is a  Jordan curve that is convex.  Using inequality \eqref{Wulff Gage inequality}, we obtain
\begin{equation}\label{chain_Gage}
\dfrac{\kappa}{A(K)}P_F(K)\leq\ds\int_{\de K} (k^F_{\de K}(x)^2F(\nu_{\de K}(x) ) \;d\mathcal{H}^1(x)\leq \left(k^F_{\max}(\de K)\right)^2 P_F(K)
\end{equation}
and so inequality \eqref{main} immediately follows.

\noindent
\underline{\textbf{Step 3: The inequality holds for general curves} }
\noindent
Using the anisotropic curvature flow, the case of  general curves will be reduced to the case of convex curves, in the same spirit of \cite{P}.
We set $A_0:=A(K)$ and we prove that $k^F_{\max}(\gamma)\geq \sqrt{A_0/\kappa}:=C$,  for every admissible $\gamma$.
By contradiction, there exists a smooth Jordan curve $\overline {\gamma}$ (not convex) such that 
\begin{equation}\label{contradiction}
k^F_{\max}(\overline {\gamma})<C.
\end{equation}
Let $ u(\cdot,t )$, with $t\in[0,T]$, be the  family of curves  evolving by anisotropic curvature flow with $u(\cdot, 0)=\overline {\gamma}(\cdot)$; so that  at time $t=T$  the area enclosed by $u(\cdot,T)$ is $0$. We consider the family
\begin{equation*}
U(\cdot, t):=f(t) u(\cdot,t),
\end{equation*}
where $f$ is a non-negative function chosen in  such a way that every curve of the family  $U(\cdot, t)$ encloses constant area. Therefore,  
\begin{equation*}
f(t)=\sqrt{\dfrac{A_0}{A(t)}},
\end{equation*}
where $A(t)$ is the area enclosed by $u_t(\cdot):=u(\cdot,t)$.   Moreover, we observe that 
\begin{equation}\label{relation_curvature}
k^F_U=\left(\dfrac{1}{f}\right) k^F_u.
\end{equation}
Recalling that we denote by $'$ the derivative with respect to $\theta$, using  \eqref{relation_curvature} and \eqref{equazione del calore anisotropa}, we obtain
\begin{multline}\label{main_computation}
	 \left(  \partial_t-\psi \partial_{ss} \right) \dfrac{\left(k^F_U\right)^2}{2}= \left(  \partial_t-\psi \partial_{ss} \right) \left[ \dfrac{A(t)}{2 A_0} \dfrac{\left(k^F_u\right)^2}{2}  \right]
	=\\=A'(t) \frac{\left(k^F_u\right)^2}{2 A_0}+\frac{A(t)}{A_0} \left(  \partial_t-\psi \partial_{ss} \right)\left(k^F_u\right)^2\leq\\ \leq A'(t)  \frac{\left(k^F_u\right)^2}{2 A_0}+\frac{A(t)}{A_0}\left[\left(  3 k_u h \phi'+h' k_u \phi   \right)\partial_s(k_u^F)^2+(k_u^F)^4\right]=\\
	=A'(t)  \frac{\left(k^F_u\right)^2}{2 A_0}+\frac{A(t)}{A_0}\left(k^F_u\right)^4+\frac{A(t)}{A_0}\left[\left(  3 k_u h \phi'+h' k_u \phi   \right)\partial_s(k_u^F)^2\right]=\\=
	\frac{A'(t) }{A(t)}(k_U^F)^2+\dfrac{A_0}{A(t)} (k_U^F)^4+\frac{A(t)}{A_0}\left[\left(  3 k_u h \phi'+h' k_u \phi   \right)\partial_s(k_u^F)^2\right].
\end{multline}
At this point let us introduce some useful notations; we set  $k_u^F(\theta,t):=k^F_{u_t}(\theta)$ and 
$k_U^F(\theta,t):=k^F_{U_t}(\theta)$.
Now, by \eqref{contradiction}, there exists $M\in(0,C)$ such that $k^F_{\overline {\gamma}}(\theta)<M$ for every $\theta\in\mathbb{S}^1$ and we want to show that for every $\theta\in\mathbb{S}^1$ and for every $t$ 
\begin{equation}\label{final}
k^F_{u}(\theta,t)<M<C. 
\end{equation}
In order to prove \eqref{final}, we proceed again by contradiction, assuming that there exists $t^*\in(0,T)$ for which it is possible to find a $\theta^*$ such that 
$ k^F_U(\theta^*,t^*)=M$. This means that $\theta^*$ is a maximum for $k_U^F(\cdot,t^*)$ and, as a consequence, it is a maximum also for $k_u^F(\cdot,t^*)$. So, taking into account that at a maximal point $\partial_s(k^F_u)$ vanishes and $\left(  k^F_u \right)_{ss}(\theta^*,t^*)$ is non-positive, from \eqref{main_computation} we obtain that
\begin{equation}\label{min}
 \left(  \partial_t-\psi \partial_{ss} \right) \dfrac{\left(k^F_U(\theta^*,t^*)\right)^2}{2}\leq \frac{M^2}{A'(t^*)}\left(\frac{A'(t^*)}{2}+A_0 M^2\right).
\end{equation}
Using  $\eqref{derar}$, we have that 
\begin{multline}
A'(t^*)=-\ds\int_{u_{t^*}}  F(\nu_{u_{t^*}}(s,t^*)) k^F_{u_{t^*}}(s,t^*) ds=-\int_{\de\Omega_{t^*}}F(\nu_{u_{t^*}}(x)) k^F_{u_{t^*}}(x)\;d \mathcal{H}^1(x)\\\leq -a D\int_{u_{t^*}}k_{u_{t^*}}(x)d \mathcal{H}^1(x)=-2\pi a D,
\end{multline}
wher $\Omega_{t^*}$ is the set enclosed by $u_{t^*}$. In the last inequality we have used the following facts: that, for every unit vector $v$, $F(v)\geq a $, the fact that the anisotropic curvature is controlled from above by the classical curvature since $F$ is elliptic (see Remark $2.7$), and, finally, the Gauss-Bonnet theorem.
\noindent
As a consequence, 
\begin{equation}\label{boh}
 \left(  \partial_t-\psi \partial_{ss} \right) \dfrac{\left(k^F_U(\theta^*,t^*)\right)^2}{2}\leq -\dfrac{A_0 M^2}{A(t^*)}\left(\frac{\pi a D}{A_0}-M^2\right)<0,
\end{equation}
since we can assume, using a suitable scaling, that $A_0$ is such that $\frac{\pi a D}{A_0}=C$.
Now, having  $\partial_{ss}\left(k^F_U(\theta^*,t^*)\right)^2/2<0$, from \eqref{boh}, we have that
\begin{equation*}
	\partial_t \left(k^F_U(\theta^*,t^*)\right)^2<0,
\end{equation*}
and so
\begin{equation}\label{finally}
\partial_t \left(k^F_U(\theta^*,t^*)\right)<0.
\end{equation}
It follows that $k^F_U(\theta^*,t^*-\epsilon)>M$, for $\epsilon>0$ small enough, which contradicts the choice of $t^*$. 
In this way we have proved  \eqref{final}.

\noindent
Now, for the properties of the anisotropic curvature flow (see Remark $2.8$ and the reference therein), we know that for some $\tau>0$ the curve $U(\cdot,\tau)$ is convex and therefore, thanks to Step $2$, we have that for some $\theta\in[0,2\pi]$
\begin{equation}
k^F(\theta,\tau)\geq C,
\end{equation}
that contradicts \eqref{final}, concluding the proof.

\end{proof}

\small{

}

\end{document}